\documentclass[a4paper,reqno]{amsart}
\usepackage{mathtools,amsmath,amssymb,amsfonts,amsthm}
\usepackage[a4paper]{geometry}
\usepackage{hyperref}
\usepackage[T1]{fontenc} 
\usepackage[utf8]{inputenc}

\def\R{\mathbb{R}}
\def\Z{\mathbb{Z}}
\renewcommand{\H}{\mathcal{H}}
\newcommand{\N}{\mathbb{N}}
\newcommand{\Q}{\mathcal{Q}}
\newcommand{\Bad}{\mathcal{B}}
\newcommand{\Good}{\mathcal{G}}

\newcommand{\Sn}{{\mathbb{S}^{d-1}}}
\newcommand{\Rn}{{\R}^d}
\newcommand{\Ln}{{\mathcal{L}}^d}
\newcommand{\xy}{^\xi_y}
\newcommand{\sm}{\setminus}
\newcommand{\dx}{\, \mathrm{d} x}
\newcommand{\Mb}{\mathcal{M}_b}
\renewcommand{\dh}{\, \mathrm{d} \mathcal{H}^{d-1}}
\newcommand{\hn}{\mathcal{H}^{d-1}}
\newcommand{\ho}{\mathcal{H}^0}
\newcommand{\Mnn}{{\mathbb{M}^{d\times d}_{sym}}}

\newcommand{\I}{\mathrm{I}^{\sigma,\xi}_y}
\newcommand{\ol}{\overline}

\theoremstyle{plain}
\begingroup
\theoremstyle{plain}
\newtheorem{theorem}{Theorem}[section]

\newtheorem{lemma}[theorem]{Lemma}
\theoremstyle{definition}
\newtheorem{definition}[theorem]{Definition}
\theoremstyle{remark}
\newtheorem{remark}[theorem]{Remark}

\endgroup

\numberwithin{equation}{section}

\title{A general compactness theorem in $G(S)BD$}
\author{Antonin Chambolle}
\address{CEREMADE, CNRS, Universit\'e Paris-Dauphine PSL, France
  and Mokaplan (INRIA/CNRS/PSL)}
\email[Antonin Chambolle]{antonin.chambolle@ceremade.dauphine.fr}
\author{Vito Crismale}
\address{Dipartimento di Matematica Guido Castelnuovo, Piazzale Aldo Moro 5, 00185 Roma, Italy}
\email[Vito Crismale]{vito.crismale@mat.uniroma1.it}

\begin{document}
\maketitle
\begin{abstract}
  We give a new, simpler proof of a compactness result in $GSBD^p$, $p>1$,
  by the same authors, which is also valid in $GBD$ (the case $p=1$),
  and shows that bounded sequences converge a.e., after removal of a suitable
  sequence of piecewise infinitesimal rigid motions, subject to a fixed partition.
\end{abstract}

\section{Introduction}

\textit{Generalized (special) functions with bounded deformation} ($G(S)BD$) have
been introduced by G.~Dal Maso~\cite{DM13} in order to properly tackle
\textit{free discontinuity problems}~\cite{DGA,AFP} in linearized elasticity,
and in particular the minmization of the Griffith functional
\begin{equation}\label{eq:Griffith}
  \int_{\Omega\setminus K} \mathbb{C}e(u):e(u)dx + \gamma\hn(K),
\end{equation}
introduced in~\cite{FM} to model and approximate brittle fracture growth in linear
elastic materials. In this functional, $\Omega\subset\R^d$ is a bounded
$d$-dimensional domain (in practice $d\in\{2,3\}$), $u$ a vectorial displacement,
expected to be smooth, with symmetrized gradient $e(u)=(Du+Du^T)/2$,
except across a $(d-1)$-dimensional fracture set $K$.
The tensor $\mathbb{C}$ contains the physical constants of the problem, and
defines a positive definite quadratic form on symmetric tensors, while $\gamma>0$
is the toughness of the material. Showing existence to minimizers to this
functional has been a difficult task, developed over many years. 
The situation mostly evolved after~\cite{DM13} introduced for the first time
a reasonable energy space for a weak form of~\eqref{eq:Griffith}, where
$K$ is replaced with $J_u$, the intrinsic jump set of $u$~\cite{DelNin}.
Existence results could then be proved~\cite{BelCosDM98, FriPWKorn, CC-JEMS,CFI19,CCI19,CC-CalcVar} for weak,
then ``strong'' minimizers (that is, for the original problem in $(u,K)$). Most
of these works rely upon a rigidity result for displacements with small jumps,
established in~\cite{CCF16}.

In particular, the main result in~\cite{CC-JEMS} is a compactness result in $GSBD^p$, the
subspace of $GSBD$ (which is defined precisely in Section~\ref{sec:prelim})
of displacements with $p$-integrable symmetrized gradient
and jump set of finite surface. In this result, a sequence which is bounded
in energy (roughly, \eqref{eq:Griffith}, with the Lagrangian replaced with $|e(u)|^p$) will converge up to subsequences
either to a $GSBD^p$ limit $u(x)$ or to $+\infty$ (with some
appropriate semicontinuity properties). This is not really an issue
for the study of~\eqref{eq:Griffith}, since replacing $u$ with $0$ where
it is infinite, we recover that the limit of a minimizing sequence is
a minimizer.

However, it was observed in~\cite{Fri19, ChaCri20AnnSNS} that this compactness result is not
sufficient for studying more general, non-homogeneous variational problems,
where the Lagrangian is not minimal at $0$. In that case, one has to study more
finely what happens in the ``infinity'' set. Following similar (yet far more precise) results
in the scalar case~\cite{Fri19}, the authors could show in~\cite{ChaCri20AnnSNS} a
more complete compactness result, and in particular the existence of a Caccioppoli
partition where, on each set of the partition, the sequence converges to a finite
limit after substraction of a suitable sequence of infinitesimal rigid motions
(affine functions with skew-symmetric gradients).

In addition, S.~Almi and E.~Tasso~\cite{AlmiTasso} recently extended~\cite{CC-JEMS},
with a different proof, to sequences merely bounded in $GBD$ (roughly, the
case $p=1$ in~\cite{CC-JEMS}), while the proof in~\cite{ChaCri20AnnSNS}, relying on a fine
result of~\cite{CagChaSca20} valid only for $p>1$, would not work in $GBD$.

The purpose of this note is to give an alternative proof of the main compactness result
of~\cite{ChaCri20AnnSNS},
which does not rely on~\cite{CagChaSca20} and is also valid in $GBD$, thus permitting to deal with non-homogeneous problems also in this framework. 
Precisely, we show the compactness Theorem~\ref{thm:main} below (the notation is made precise in Section~\ref{sec:prelim}). The proof of this result is quite simpler, in a sense,
than in~\cite{ChaCri20AnnSNS}, yet also more interesting. It only relies on a suitable version of the
approximate Poincar\'e-Korn inequality of~\cite{CCF16} proven in Theorem~\ref{th:KP}, which asserts that
the energy controls how far a function is to rigid motions
(hence to finite-dimensional), combined with a multiscale construction. We hope that this scheme can be useful for other purposes. We observe for instance that, combined with
the celebrated extension method of Nitsche~\cite{Nitsche81},
a simplified version or this proof allows to easily deduce
 Rellich-type theorems in $BD$~\cite{TemStr,Suquet1981,Tem,KohnTemam}.
\begin{theorem}\label{thm:main}
Let $\Omega\subset \R^d$ be a bounded domain and let $u_k \in GBD(\Omega)$ be such that
\begin{equation}
\sup_{k\in \N} \widehat{\mu}_{u_k}(\Omega)<+\infty.
\end{equation}
Then there exist a subsequence, not relabelled, 
a Caccioppoli partition $\mathcal{P}=(P_n)_n$ of $\Omega$, a sequence of piecewise rigid motions   $(a_k)_k$ with
 \begin{subequations}\label{eqs:0808222034}
\begin{equation}\label{2302202219}
a_k=\sum_{n\in \N} a_k^n \chi_{P_n},
\end{equation}
\begin{equation}\label{2202201909}
|a_k^n(x)-a_k^{n'}(x)| \to +\infty \quad \text{for }\Ln\text{-a.e.\ }x\in \Omega, \text{ for all }n\neq {n'},
\end{equation}
\end{subequations}
 and $u\in GBD(\Omega)$ such that 
\begin{subequations}\label{eqs:0203200917}
\begin{align}
u_k-a_k &\to u \quad \Ln\text{-a.e.\ in }\Omega, \label{2202201910}\\
\hn(\partial^* \mathcal{P} \cap \Omega)&\leq \lim_{\sigma\to +\infty}\liminf_{k\to \infty} \,\hn(J^\sigma_{u_k}).\label{eq:sciSaltoInf}
\end{align}
\end{subequations}
\end{theorem}
If in addition $(u_k)_k$ is bounded in $GSBD^p(\Omega)$, $p>1$
(that is, \eqref{eq:bdGSBDp} below holds), following~\cite{ChaCri20AnnSNS} one
obtains in addition to the last estimate:
\addtocounter{equation}{-1}
\begin{subequations}
  \begin{align}\label{eq:sciSaltoInfp}\addtocounter{equation}{2}
  \hn((\partial^*\mathcal{P}\cup J_u)\cap\Omega)
  \le \liminf_{k\to \infty} \,\hn(J_{u_k}),
  \end{align}
\end{subequations}
see Remark~\ref{rem:GSBDp} in Section~\ref{sec:lsc}. 


The plan of the note is as follows: we first define properly the notions which
are useful for this work (Sec.~\ref{sec:prelim}). Then, in Section~\ref{sec:rigidity}
we show that a partial rigidity
result of~\cite{CCF16} is also valid in $GBD$, without further integrability
assumption. The following section is devoted to the proof of Theorem~\ref{thm:main}.
Thanks to the rigidity result, we build an appropriate Caccioppoli
partition which will satisfy the thesis of the Theorem in Section~\ref{subsec:partition}.
We end up proving the compactness~\eqref{2202201910}
(Sec.~\ref{sec:compact}) and the lower-semicontinuity~\eqref{eq:sciSaltoInf} (Sec.~\ref{sec:lsc}).

\section{Preliminaries}

\subsection{Notation}\label{sec:prelim}
Given $\Omega \subset \Rn$ open, we   use the notation  
$L^0(\Omega;\R^m)$  for  the space of $\Ln$-measurable functions $v \colon \Omega \to \R^m$, endowed with the topology of convergence in measure.
For any locally compact subset $B  \subset \Rn$, (i.e.\ any point in $B$ has a neighborhood contained in a compact subset of $B$),
the space of bounded $\R^m$-valued Radon measures on $B$ [respectively, the space of $\R^m$-valued Radon measures on $B$] is denoted by $\mathcal{M}_b(B;\R^m)$ [resp., by $\mathcal{M}(B;\R^m)$]. If $m=1$, we write $\mathcal{M}_b(B)$ for $\mathcal{M}_b(B;\R)$, $\mathcal{M}(B)$ for $\mathcal{M}(B;\R)$, and $\mathcal{M}^+_b(B)$ for the subspace of positive measures of $\mathcal{M}_b(B)$. For every $\mu \in \mathcal{M}_b(B;\R^m)$, its total variation is denoted by $|\mu|(B)$.

 We say that $v\in L^1(\Omega)$ is a \emph{function of bounded variation} on $\Omega$, and we write $v\in BV(\Omega)$, if $\mathrm{D}_i v\in \mathcal{M}_b(\Omega)$ for  $i=1,\dots,n$,  where $\mathrm{D}v=(\mathrm{D}_1 v,\dots, \mathrm{D}_n v)$ is its distributional  derivative.  A vector-valued function $v\colon \Omega\to \R^m$ is in $BV(\Omega;\R^m)$ if $v_j\in BV(\Omega)$ for every $j=1,\dots, m$.
The space $BV_{\mathrm{loc}}(\Omega)$ is the space of $v\in L^1_{\mathrm{loc}}(\Omega)$ such that $\mathrm{D}_i v\in \mathcal{M}(\Omega)$ for $i=1,\dots,d$. 

We call \emph{infinitesimal rigid motion} any affine function with skew-symmetric gradient and \emph{piecewise rigid motion} any function of the form $\sum_{j\in \N} a_j \chi_{P_j}$, where $(P_j)_j$ is a Caccioppoli partition of $\Omega$ (that
is, a partition into sets of finite perimeters, with finite total perimeter)
and any $a_j$ is an infinitesimal rigid motion.

Fixed $\xi \in \Sn$, we let
\begin{equation}\label{eq: vxiy2}
\Pi^\xi:=\{y\in \Rn\colon y\cdot \xi=0\},\qquad B^\xi_y:=\{t\in \R\colon y+t\xi \in B\} \ \ \ \text{ for any $y\in \Rn$ and $B\subset \Rn$},
\end{equation}
and for every function $v\colon B\to  \Rn$ and $t\in B^\xi_y$, let
\begin{equation}\label{eq: vxiy}
v^\xi_y(t):=v(y+t\xi),\qquad \widehat{v}^\xi_y(t):=v^\xi_y(t)\cdot \xi.
\end{equation}
Moreover, let $\Pi^\xi(x):=x-(x\cdot \xi) \xi \in \xi^\perp=\{y\in \Rn \colon y\cdot \xi=0\}$ for every $x\in\Rn$.

\begin{definition}[``$GBD$'' \cite{DM13}]\label{def:GBD}
Let $\Omega\subset \Rn$ be a  bounded open set, and let  $v \in L^0(\Omega;\Rn)$.   Then $v\in GBD(\Omega)$ if there exists $\lambda_v\in \mathcal{M}^+_b(\Omega)$ such that  one of the following equivalent conditions holds true 
 for every 
$\xi \in \Sn$: 
\begin{itemize}
\item[(a)] for every $\tau \in C^1(\R)$ with $-\tfrac{1}{2}\leq \tau \leq \tfrac{1}{2}$ and $0\leq \tau'\leq 1$, the partial derivative $\mathrm{D}_\xi\big(\tau(v\cdot \xi)\big)=\mathrm{D}\big(\tau(v\cdot \xi)\big)\cdot \xi$ belongs to $\mathcal{M}_b(\Omega)$, and for every Borel set $B\subset \Omega$ 
\begin{equation*}
\big|\mathrm{D}_\xi\big(\tau(v\cdot \xi)\big)\big|(B)\leq \lambda_v(B);
\end{equation*}
\item[(b)] $\widehat{v}^\xi_y \in BV_{\mathrm{loc}}(\Omega^\xi_y)$ for $\hn$-a.e.\ $y\in \Pi^\xi$, and for every Borel set $B\subset \Omega$ 
\begin{equation*}
\int_{\Pi^\xi} \Big(\big|\mathrm{D} {\widehat{v}}_y^\xi\big|\big(B^\xi_y\setminus J^1_{{\widehat{v}}^\xi_y}\big)+ \mathcal{H}^0\big(B^\xi_y\cap J^1_{{\widehat{v}}^\xi_y}\big)\Big)\dh(y)\leq \lambda_v(B),
\end{equation*}
where
$J^1_{{\widehat{u}}^\xi_y}:=\left\{t\in J_{{\widehat{u}}^\xi_y} : |[{\widehat{u}}_y^\xi]|(t) \geq 1\right\}$.
\end{itemize} 
The function $v$ belongs to $GSBD(\Omega)$ if $v\in GBD(\Omega)$ and $\widehat{v}^\xi_y \in SBV_{\mathrm{loc}}(\Omega^\xi_y)$ for 
every
$\xi \in \Sn$ and for $\hn$-a.e.\ $y\in \Pi^\xi$.
\end{definition} 

For $v\in GBD(\Omega)$, denoting by
\begin{equation}\label{1707221145}
(\widehat{\mu}_v)\xy (B):= |\mathrm{D} {\widehat{v}}_y^\xi\big|\big(B\setminus J^1_{{\widehat{v}}^\xi_y}\big)+ \mathcal{H}^0\big(B\cap J^1_{{\widehat{v}}^\xi_y}\big)\quad\text{for every $B\subset \Omega\xy$ Borel}
\end{equation}
($(\widehat{\mu}_v)\xy \in \Mb^+(\Omega\xy)$ for every $\xi \in \Sn$ and $\hn$-a.e.\ $y\in \Pi_\xi$) and by
\begin{equation}\label{1707221151}
(\widehat{\mu}_v)^\xi(B):=\int_{\Pi^\xi} (\widehat{\mu}_v)\xy(B\xy)\dh(y) \quad\text{for every $B\subset \Omega$ Borel},
\end{equation}
it holds that $(\widehat{\mu}_v)^\xi\in \Mb^+(\Omega)$, $(\widehat{\mu}_v)^\xi\leq \lambda_v$ for any $\lambda_v$ satisfying condition (b) of Definition~\ref{def:GBD} and that
\begin{equation}\label{1707221155}
\widehat{\mu}_v(B):=\sup_k \sup \Big\{ \sum_{i=1}^k (\widehat{\mu}_v)^{\xi_i}(B_i) \colon (\xi_i)_i\subset \Sn, \, B_1, \dots, B_k \subset B, B_i \cap B_j=\emptyset \, \forall i \neq j  \Big\}
\end{equation}
is the smallest measure $\lambda_v$ that satisfies condition (b) of Definition~\ref{def:GBD}.

Every $v\in GBD(\Omega)$ has an \emph{approximate symmetric gradient} $e(v)\in L^1(\Omega;\Mnn)$ such that for every $\xi \in \Sn$ and $\hn$-a.e.\ $y\in\Pi^\xi$ there holds
\begin{equation}\label{3105171927}
 e(v)(y +  t\xi)   \xi   \cdot \xi= (\widehat{v}^\xi_y)'(t)  \quad\text{for } \mathcal{L}^1\text{-a.e.\ }   t \in  \Omega^\xi_y;
\end{equation}
the \emph{approximate jump set} $J_v$ is still countably $(\hn,d{-}1)$-rectifiable (\textit{cf.}~\cite[Theorem~6.2]{DM13} and \cite{DelNin}) and may be reconstructed from its slices through the identity
\begin{equation}\label{slicingSaltoGSBD}
(J^\xi_v)^\xi_y= J_{\widehat{v}^\xi_y} \quad\text{and}\quad v^\pm(y+t\xi)  \cdot \xi = (\widehat{v}^\xi_y)^\pm(t) \ \text{ for }t\in (J_v)^\xi_y,
\end{equation}
where $J^\xi_v:=\{x\in J_v \colon [v]\cdot \xi \neq 0\}$ (it holds that $\hn(J_v\sm J_v^\xi)=0$ for $\hn$-a.e.\ $\xi \in \Sn$).
For every $\sigma>0$ we also denote
\begin{equation}\label{def:Jvsigma}
J_v^\sigma:=\{ x\in J_v \colon |[v]|<\sigma\}.
\end{equation}
By \eqref{slicingSaltoGSBD}, for every $\sigma>0$, every $\xi \in \Sn$, and $\hn$-a.e.\ $y\in \Pi_\xi$
\begin{equation}\label{1807221039}
J_{\widehat{v}^\xi_y}^\sigma \subset (J_v^\sigma)\xy,
\end{equation}
where $J_{\widehat{v}^\xi_y}^\sigma=\{ t\in J_{\widehat{v}^\xi_y}\colon |[\widehat{v}^\xi_y]|<\sigma\}$.

We recall from \cite{ChaCri20AnnSNS} the following lemma on piecewise rigid motions.
 \begin{lemma}\label{le:2502200924}
Let  $(\mathcal{P}_j)_j$ be a Caccioppoli partition  and  let $(a_h)_h$ be a sequence of  piecewise  rigid motions such that \eqref{2302202219} and \eqref{2202201909} hold. Then 
for $\hn$-a.e.\ $\xi \in \Sn$
\begin{equation}\label{2502200925}
|(a_h^j-a_h^i)(x)\cdot \xi| \to +\infty \quad\text{as }h\to +\infty \quad \text{for }\Ln\text{-a.e.\ }x\in \Omega, \text{ for all }i\neq j.
\end{equation}
\end{lemma}

\subsection{Rigidity in $GBD$}\label{sec:rigidity}
The following result is obtained in the footsteps of Proposition~2.1 in~\cite{CCF16}. Let
 $Q_\delta = (-\delta/2,\delta/2)^d$.
\begin{theorem}\label{th:KP} There exist $c>0$ such that
  for any $\delta>0$, $u\in GSBD(Q_\delta)$, there
  exists $\omega\subset Q_\delta$ with $|\omega|\le c\delta\H^{d-1}(J_u^1)$
  and an infinitesimal rigid motion $a$
  such that
  \[
    \int_{Q_\delta\setminus\omega} |u-a|\dx \le c\delta \,\widehat{\mu}_u(Q_\delta\setminus J_u^1).
  \]
\end{theorem}
\begin{proof}
We sketch the proof, highlighting the modifications with respect to \cite[Proposition~2.1]{CCF16}.

As in \cite{CCF16}, we may assume by a rescaling argument $\delta=1$ (and write $Q$ for $Q_1$), and that $\hn(J_u^1)\leq \frac{1}{32 d^3}$, otherwise it is enough to take $\omega=Q$, $a=0$, $c=32 d^3$.
We define the function $T\colon \Rn \times \Sn \times \R \to \R$ by
\begin{equation}\label{analoga26}
T(x,\xi, t):=
\begin{dcases}
1 \quad &\text{if }x\in Q, \, x+t\xi \in Q \text{ and } x+[0,t] \xi \cap J^1_u\neq\emptyset,\\
0 \quad & \text{otherwise}.
\end{dcases}
\end{equation}
By definition of distributional derivative it holds that
\begin{equation}\label{analoga25}
\xi \cdot \big(u(x+t\xi)-  u(x) \big)= \int_{x\cdot \xi}^{x\cdot \xi + t}\mathrm{D} \widehat{u}\xy(s) \,\mathrm{d}s , \quad y:= \Pi^\xi(x)
\end{equation}
and $\mathrm{D} \widehat{u}\xy \leq (\widehat{\mu}_u)\xy$ on $[x\cdot \xi, x\cdot \xi + t]$ if $T(x,\xi, t)=0$ and $x,\,x+t\xi \in Q$ (recall \eqref{1807221039})
at least for a.e.~$x\in Q$ and $t\in\R$.
We remark that \eqref{analoga26} is the analogue of \cite[definition (2.6)]{CCF16} when replacing $J_u$ with $J_u^1$, and that \eqref{analoga25} is the analogue of \cite[equation (2.5)]{CCF16}.
Therefore, exactly as in \cite{CCF16}, one obtains that
\begin{equation*}
\int_{\Rn}\int_{\R} T(x,\xi, t) \,\mathrm{d}t\dx \leq 4d \hn(J_u^1) \quad \text{for any }\xi \in \Sn
\end{equation*}
and that there exists $t_* \in (1/2, 1)$ (fixed for the rest of the proof) and $q_1 \subset q$ with $\Ln(q_1)/\Ln(q)>3/4$ such that, defining
\begin{subequations}\label{eqs:analoga28}
\begin{equation}\label{1707221737}
z_i(z_0):=z_0+ t_* e_i \quad \text{for all }i=1,\dots, d, \qquad E_{z_0}:=\Big(\bigcup_{i=1}^d [z_0, z_i(z_0)] \cup \bigcup_{1\leq i < j \leq d} [z_i(z_0), z_j(z_0)] \Big)
\end{equation}
it holds
\begin{equation}\label{analoga28}
\text{for $z_0 \in q_1$: }  E_{z_0} \cap J_u^1 = \emptyset, \, E_{z_0}\subset Q.
\end{equation}
\end{subequations}

For $t_*$ fixed above and any $z_0\in q$ let us consider
\begin{equation*}
\begin{split}
F(z_0)&:=\sum_{0\leq i < j\leq n} |\mathrm{D} \widehat{u}^{\xi_{i,j}}_{y_{i,j}(z_0)}|([z_i(z_0) \cdot \xi_{i,j}, z_j(z_0)\cdot \xi_{i,j}]), \\
& \xi_{i,j}:=\frac{z_i(z_0)-z_j(z_0)}{|z_i(z_0)-z_j(z_0)|},\, y_{i,j}(z_0):=\Pi^{\xi_{i,j}}(z_i(z_0))=z_i(z_0)-(z_i(z_0) \cdot \xi_{i,j}) \xi_{i,j}.
\end{split}
\end{equation*}
We notice that $\xi_{i,j}=\frac{e_i-e_j}{\sqrt{2}}$ if $i \neq 0$, while $\xi_{i,j}=e_j$ if $i=0$.
Fixed $i \neq j$, we integrate for $z_0\in q$ as $\widetilde{z}_0=\Pi^{\xi_{i,j}}(z_0)$ ranges in $\xi_{i,j}^\perp$ and $z'_0=z_0\cdot \xi_{i,j}$ ranges in $q^{\xi_{i,j}}_{\Pi^{\xi_{i,j}}(z_0)}$, using Fubini's Theorem. 
Moreover, if $\Pi^{\xi_{i,j}}(z_0)$ is fixed to a value $\widetilde{z}_0\in \Pi^{\xi_i,j}$, also $y_{i,j}(z_0)$ is fixed and equal to 
\[
\widehat{z}_0:= \Pi^{\xi_{i,j}}(z_0)+t_* \Pi^{\xi_{i,j}}(e_i)=\widetilde{z}_0+t_* \Pi^{\xi_{i,j}}(e_i),
\]
 so that in such a case
$[z_i(z_0) \cdot \xi_{i,j}, z_j(z_0)\cdot \xi_{i,j}]\subset  \widehat{z}_0 + \R \xi_{i,j} \cap Q$
and 
\[
|\mathrm{D}\widehat{u}^{\xi_{i,j}}_{y_{i,j}(z_0)}|([z_i(z_0) \cdot \xi_{i,j}, z_j(z_0)\cdot \xi_{i,j}]) \leq (\widehat{\mu}_u)^{\xi_{i,j}}_{\widehat{z}_0}((Q\sm J_u^1)^{\xi_{i,j}}_{\widehat{z}_0})
\]
 regardless of the value of $z'_0=z_0\cdot \xi_{i,j}$ (satisfying $\widetilde{z}_0+z'_0 \xi_{i,j}\in q$ since we integrate over $z_0\in q$). It follows that (notice that $\mathcal{L}^1(\{s\in \R\colon\widetilde{z}_0+z'_0 \xi_{i,j}\in Q\}) \leq \sqrt{2}$)
\begin{equation}\label{1707221755}
\begin{split}
  \int_q |\mathrm{D} \widehat{u}^{\xi_{i,j}}_{y_{i,j}(z_0)}|([z_i(z_0) \cdot \xi_{i,j}, & z_j(z_0)\cdot \xi_{i,j}]) \,\mathrm{d}z_0
  \\
  &= \hspace{-1em}\int\limits_{\widetilde{z}_0=\Pi^{\xi_{i,j}}_{z_0}\in \xi_{i,j}^\perp} \hspace{-2em} \dh(\widetilde{z}_0) \hspace{-2em}\int\limits_{(\widetilde{z}_0+\R \xi_{i,j})\cap q}  \hspace{-2em}  |\mathrm{D}   \widehat{u}^{\xi_{i,j}}_{y_{i,j}(z_0)}|([z_i(z_0) \cdot \xi_{i,j}, z_j(z_0)\cdot \xi_{i,j}]) \, \mathrm{d}z'_0 \\
& \leq \sqrt{2}  \int\limits_{\widetilde{z}_0\in \xi_{i,j}^\perp}   (\widehat{\mu}_u)^{\xi_{i,j}}_{\widehat{z}_0}((Q\sm J_u^1)^{\xi_{i,j}}_{\widehat{z}_0})\dh(\widetilde{z}_0)\\
& \leq \sqrt{2} (\widehat{\mu}_u)^{\xi_{i,j}}(Q\sm J_u^1).
\end{split}
\end{equation}
Summing \eqref{1707221755} over $0\leq i< j\leq d$, we get
\begin{equation*}
\int_q F(z_0) \,\mathrm{d}z_0 \leq \sqrt{2} (d+1)^2 \widehat{\mu}_u(Q\sm J_u^1).
\end{equation*}
and there exists $q_2 \subset q$ with $\Ln(q_2)/\Ln(q)>3/4$ such that for every $z_0\in q_2$
\begin{equation}\label{analoga29}
F(z_0) \leq 4 \sqrt{2} (d+1)^2 \widehat{\mu}_u(Q\sm J_u^1).
\end{equation}
This is the analogue of \cite[condition (2.8)]{CCF16}. 
At this stage, following \cite{CCF16}, it holds that for any $z_0$ satisfying \eqref{analoga28} and \eqref{analoga29} the affine map $a\colon \Rn \to \Rn$ such that $a(z_i(z_0))=u(z_i(z_0))$ for all $i=0, \dots, d$ satisfies
\begin{equation}\label{analoga211}
|e(a)|\leq c \widehat{\mu}_u(Q\sm J_u^1).
\end{equation}
Arguing exactly as in Step~3 of the proof of Proposition~2.1 in \cite{CCF16} we find a set $q_3\subset q$ with $\Ln(q_3)/\Ln(q)>3/4$ such that if $z_0\in q_3$ then
\begin{equation}\label{analoga213}
\begin{split}
&\text{ \eqref{analoga25} holds for any $[z_i(z_0),y]$ for $y \in Q\sm \omega$, where }\\
&\omega=\bigcup_{i=0}^n \omega_{(i)},\qquad \omega_{(i)}:=\{y\in Q \colon y=z_i+t \xi \text{ with }T(z_i, \xi, t)=1\}.
\end{split}
\end{equation}
With \eqref{analoga25} and the fact that for any $y$ there are $d$ points in $\{z_0, \dots, z_d\}$ such that the simplex generated by those and $y$ has volume at least $t_*/(d+1)!$, \eqref{analoga213} implies that
\begin{equation}\label{analoga215}
\begin{split}
|w(y)|\leq & c \sum_{i=0}^n |\mathrm{D} \widehat{w}^{\xi_{i,y}}_{x_{i,y}(z_0)}|([z_i(z_0) \cdot \xi_{i,y}, y\cdot \xi_{i,y}]) \quad \text{for all }y\in Q \sm \omega, \text{ where }\\
&
w:=u-a,\quad \xi_{i,y}:= \frac{y-z_i(z_0)}{|y-z_i(z_0)|},\quad x_{i,y}(z_0):= \Pi^{\xi_{i,y}}(z_i(z_0))
\end{split}
\end{equation}
Let us consider the quantity
\begin{equation}\label{1807221055}
H(z_0):= \sum_{i=0}^n H_i(z_0), \quad H_i(z_0):= \int_{Q\sm \omega} |\mathrm{D} \widehat{w}^{\xi_{i,y}}_{x_{i,y}(z_0)}|([z_i(z_0) \cdot \xi_{i,y}, y\cdot \xi_{i,y}]) \,\mathrm{d}y.
\end{equation}
In the following we prove that there exists $q_4\subset q_3$ with $\Ln(q_4)/\Ln(q_3)>3/4$ such that $H(z_0)$ is controlled by $\widehat{\mu}_u(Q\sm J_u^1)$ times a constant $c$ depending only on $d$, for every $z_0\in q_4$.
Together with \eqref{analoga211} and \eqref{analoga215} this will conclude the proof.

By Fubini's Theorem, we have that for every $i=0, \dots, d$
\begin{equation*}
\int_{q_3} H_i(z_0) \,\mathrm{d}z_0 = \int_{Q\sm \omega} \int_{q_3}  |\mathrm{D} \widehat{w}^{\xi_{i,y}}_{x_{i,y}(z_0)}|([z_i(z_0) \cdot \xi_{i,y}, y\cdot \xi_{i,y}])    \,\mathrm{d}z_0 \,\mathrm{d}y.
\end{equation*}
For fixed $y\in Q\sm \omega$ we argue similarly to what done to prove \eqref{1707221755}, namely we integrate for $z_0\in q_3$ as $\widetilde{z}_0=\Pi^{\xi_{i,y}}(z_0)$ ranges in $\xi_{i,y}^\perp$ and $z'_0=z_0 \cdot \xi_{i,y}$ ranges in $(q_3)_{\widetilde{z}_0}^{\xi_{i,y}}$, using Fubini's Theorem. Then, given $\widetilde{z}_0$, we have that $x_{i,y}(z_0)= \widetilde{z}_0 + t_*\Pi^{\xi_{i,y}}(e_i)=: \widehat{z}_0$ and that
\[
|\mathrm{D} \widehat{w}^{\xi_{i,y}}_{x_{i,y}(z_0)}|([z_i(z_0) \cdot \xi_{i,y}, y\cdot \xi_{i,y}]) \leq (\widehat{\mu}_w)^{\xi_{i,y}}_{\widehat{z}_0}((Q\sm J_u^1)^{\xi_{i,y}}_{\widehat{z}_0}),
\]
regardless of the value of $z'_0$. Therefore
\begin{equation}\label{1807221143}
\begin{split}
\int_{q_3}  |\mathrm{D} \widehat{w}^{\xi_{i,y}}_{x_{i,y}(z_0)}|([z_i(z_0) \cdot \xi_{i,y}, y\cdot \xi_{i,y}])    \,\mathrm{d}z_0 & \leq \sqrt{2}\int_{\xi_{i,y}^\perp}  (\widehat{\mu}_w)^{\xi_{i,y}}_{\widehat{z}_0}((Q\sm J_u^1)^{\xi_{i,y}}_{\widehat{z}_0})    \dh(\widetilde{z}_0)\\
& \leq \sqrt{2} (\widehat{\mu}_w)^{\xi_{i,y}}(Q\sm J_u^1)  =  \sqrt{2} (\widehat{\mu}_u)^{\xi_{i,y}}(Q\sm J_u^1)
\end{split}
\end{equation}
where the equality above follows from the fact that $a$ is an infinitesimal rigid motion.
Summing \eqref{1807221143} over $i$ (and arguing as done for \eqref{analoga29}) we get that
\[
\int_{q_3} H(z_0) \,\mathrm{d}z_0 \leq c \widehat{\mu}_u(Q\sm J_u^1),
\]
thus we find $q_4\subset q_3$ with $\Ln(q_4)/\Ln(q_3)>3/4$ and $H(z_0) \leq c \widehat{\mu}_u(Q\sm J_u^1)$ for every $z_0\in q_4$; 
we conclude the proof by picking $z_0$ in $q_1 \cap q_2 \cap q_4$ (which has positive measure) and integrating \eqref{analoga215} over $y \in Q\sm \omega$.
\end{proof}
\section{Proof of the compactness theorem}
In this section we prove Theorem~\ref{thm:main}. In Subsection~\ref{subsec:partition} we construct a suitable partition $\mathcal{P}=(P_n)_n$ of $\Omega$ and a sequence piecewise rigid functions $(a_k)_k$ satisfying \eqref{eqs:0808222034}; in the next two subsections we prove the existence of $u \in GBD(\Omega)$ satisfying \eqref{2202201910} and the lower semicontinuity condition \eqref{eq:sciSaltoInf} on the surface measure of $\partial^*\mathcal{P}$, 
respectively.
\subsection{Construction of a partition}\label{subsec:partition}

For $\delta>0$, let $\Q^\delta = \{ z+(-\delta/2,\delta/2]^d:
z\in \delta\Z^d, \ z+(-\delta/2,\delta/2)^d\subset\Omega\}$.
Let $\eta>0$, small, and define
\[
  \Bad^\delta(u_k)=\{Q\in \Q^\delta: \H^{d-1}(J^1_{u_k}\cap Q)> \eta\delta^{d-1}\},\quad \Good^\delta(u_k)=\Q^\delta\setminus \Bad^\delta(u_k).
\]
By construction, one has
\begin{equation}\label{eq:estimbad}
  \left|{\textstyle\bigcup_{Q\in\Bad^{\delta}(u_k)} Q}\right|\le \frac{\delta}{\eta}\H^{d-1}(J_{u_k}^1).
\end{equation}

We fix a first value $\delta_0>0$, small enough so that $\Good^{\delta_0}(u_k)\neq\emptyset$ for
all $k\ge 1$, and for $j\ge 0$, denote
$\delta_j=\delta_0 2^{-j}$.
Upon extracting a subsequence, we may assume that $\Bad^{\delta_0}(u_k)$ is
not depending on $k$. By a diagonal argument, we may (and will) assume even
that for any $j\ge 0$, $\Bad^{\delta_j}(u_k)$ does not depend on $k$ if
$k$ is large enough (depending on $j$). We denote then $\Bad^{\delta_j}$
(and $\Good^{\delta_j}$) the corresponding limiting sets, dropping
the dependence in $u_k$.

Thanks to Theorem~\ref{th:KP}, for any $j\ge 0$ and each $Q\in \Q^{\delta_j}$, there is $\omega_k^Q\subset Q$
and $a_k^Q$, an infinitesimal rigid motion, 
such that
\begin{equation}\label{eq:rigid}
  \int_{Q\setminus\omega_k^Q} |u_k-a_k^Q|dx \le
  c\delta_j \widehat{\mu}_{u_k}(Q\setminus J_{u_k}^1)
\end{equation}
and  $|\omega_k^Q|\le c\delta_j\H^{d-1}(J_{u_k}^1\cap Q)$. In particular:
\begin{equation}\label{eq:estimomegakQ}
  Q\in\Good^{\delta_j}(u_k) \quad\Rightarrow\quad |\omega_k^Q|\le c\eta|Q|.
\end{equation}

Considering the finite family of sequences of infinitesimal rigid motions $(a_k^Q-a_k^{Q'})_k$, $\{Q,Q'\}\subset\Good^{\delta_0}$ (which does
not depend on $k$), we may assume, upon extracting
a subsequence, that  either
$|a_k^Q(x)-a_k^{Q'}(x)|\to\infty$ a.e., or $\sup_k\sup_{|x|\le 1} |a_k^Q(x)-a_k^{Q'}(x)|<+\infty$.
By a diagonal argument,
similarly, for $j\ge 1$, considering $(a_k^Q-a_k^{Q'})_k$, $\{Q,Q'\}\in \Good^{\delta_j}(u_k)$, which for $k$ large enough is $\Good^{\delta_j}$,
not depending
on $k$, we may assume the same.

We let for $j\ge 0$
\[B_j = \Big(\Omega\setminus\bigcup_{Q\in\Q^{\delta_j}}Q\Big)  \cup
  \bigcup_{l\ge j} \bigcup_{Q\in \Bad^{\delta_l}} Q.
\]
Thanks to~\eqref{eq:estimbad}, it holds that
\begin{equation}\label{eq:estimBj}
  |B_j|\le \left|\Omega\setminus{\textstyle\bigcup_{Q\in\Q^{\delta_j}}Q}\right|
  +\frac{2\delta_j}{\eta} \sup_k\widehat{\mu}_{u_k}(\Omega),
\end{equation}
so that $\lim_j  |B_j|  =0$; moreover $(B_j)_j$ is decreasing, that is
$B_{j+1}\subseteq B_j$ for all $j\ge 0$.

We define a partition $(P^j_n)_{n=1}^{N_j}$ of
$\Omega\setminus B_j$ as follows:
the sequences $(a_k^Q)$, $Q\in \Good^{\delta_j}$, for $k$ large enough so that $\Good^{\delta_j}(u_k)=\Good^{\delta_j}$,  can be grouped in equivalent classes for
the relationship $a_k^Q\sim a_k^{Q'}$ when $\sup_k\sup_{|x|\le 1} |a_k^Q(x)-a_k^{Q'}(x)|<+\infty$.
Then, we say that $Q\sim Q'$ whenever $a_k^Q\sim a_k^{Q'}$. We define, for each
equivalence class $\mathcal{C}_n$ in $\Good^{\delta_j}$, $n=1,\dots,N_j$, the set
$P^j_n = \bigcup_{Q\in\mathcal{C}_n} Q\setminus B_j$.

Observe that for any $j\ge 1$,
if $Q\in \Good^{\delta_j}$, $Q'\in \Good^{\delta_{j+1}}$ with $Q'\subset Q$ (and
 $k$ is large enough), then one has:
\[
  \int_{Q'\setminus (\omega_k^{Q}\cup \omega_k^{Q'})}|a_k^Q-a_k^{Q'}|dx
  \le c\delta\widehat{\mu}_{u_k}(Q\setminus J_{u_k}^1),
\]
so that, provided $\eta>0$ was chosen small enough
(to ensure for instance that $|\omega_k^{Q}\cup \omega_k^{Q'}|\le |Q'|/2$,
which is guaranteed if $\eta<2^{-d}/(4c)$,
\textit{cf.}~\eqref{eq:estimomegakQ}),
$\sup_k\sup_{|x|\le 1} |a_k^Q(x)-a_k^{Q'}(x)|<+\infty$.

Hence given $Q\in Q^{\delta_l}$, $Q'\in Q^{\delta_{l'}}$,
for $l'\ge l\ge j$, and such that $Q'\subset Q$ and $|Q'\setminus B_j|>0$,
one obtains by induction, $\sup_k\sup_{|x|\le 1} |a_k^Q(x)-a_k^{Q'}(x)|<+\infty$
(as all the intermediate cubes are all ``good'' at their respective scale).

It follows that:
for any smaller scales $l,l'\ge j$, and any
$Q\in \Good^{\delta_l}$, $Q'\in \Good^{\delta_{l'}}$ 
with both $Q\setminus B_j$ and $Q'\setminus B_j$ of positive measure and
contained in the same component $P_n^j$, one finds that
$\sup_k\sup_{|x|\le 1} |a_k^Q(x)-a_k^{Q'}(x)|<+\infty$.
Indeed, there are $\tilde Q,\tilde Q'\in\Good^{\delta_j}$ with $Q\subset \tilde Q$,
$Q'\subset \tilde Q'$ and $\tilde Q \sim \tilde Q'$.

In particular this shows that for $j'\ge j$ and $n\in \{1,\dots, N_j\}$,
there is $n'\in \{1,\dots , N_{j'}\}$ such that $P^j_n \subset P^{j'}_{n'}$.
We may always number the sets $(P^j_n)_n$, $j\ge 1$, according to the numbering
of $(P^{j-1}_n)_n$, so that in fact $P^j_n\subset P^{j'}_n$ for
any $j'\ge j$ and any $n\in \{1,\dots, N_j\}$. As a consequence, we
may define, for $1\le n < 1+ \sup_j N_j\in\mathbb{N}\cup\{+\infty\}$,
the set $P_n = \bigcup_j P^j_n$ (where the union starts at the first
$j$ such that $n\le N_j$).
These sets partition $\Omega\setminus\bigcap_j B_j$, hence, up to a Lebesgue-negligible
set, $\Omega$.

For each $n$, we choose an arbitrary $Q\in \Good^{\delta_j}$, at some arbitrary scale $j\ge 0$,
with $|Q\setminus B_j|>0$ and $Q\setminus B_j\subset P_n$, and we associate to $P_n$ the
subsequence $(a_k^Q)_{k}$, hence denoted $(a_k^n)_k$. It follows that
for any other such cube $Q$ at any other scale $j$, one has
$\sup_k\sup_{|x|\le 1} |a_k^n(x)-a_k^{Q}(x)|<+\infty$,
while $\lim_k |a_k^{n'}(x)-a_k^{Q}(x)|=+\infty$ a.e.~if $n'\neq n$.

\subsection{Compactness}\label{sec:compact}

We introduce the smooth,
one-Lipschitz truncations $t_\sigma(x):= \sigma\tanh(x/\sigma)$, for $\sigma>0$.
We also let $v_k = \sum_n  a_k^n\chi_{P_n}$.
Note that at the scale $j\ge 0$, one has that
\[
  v_k|_{\Omega\setminus B_j} = \sum_n  a_k^n\chi_{P_n\setminus B_j}
  = \sum_{n=1}^{N_j} a_k^n\chi_{P_n^j},
\]
showing that $v_k|_{\Omega\setminus B_j}$ is built up of
$N_j$  infinitesimal rigid motions.

For each scale $j\ge 0$ let \[w_k^j =
  \Bigg(\sum_{Q\in\Good^{\delta_j}(u_k)} a_k^Q\chi_Q - v_k\Bigg)(1-\chi_{B_j}).\]
By construction, $\sup_k \sup_{x\in\Omega\setminus B_j} |w_k^j(x)|<+\infty$ and
since $v_k|_{\Omega\setminus B_j}$ is built up of a bounded number of
affine maps, the
sequence of functions $(w_k^j)_k$ is finite-dimensional, and we may
extract a subsequence such that it converges to some limit $w^j$.
By a diagonal argument, we may assume this is true for all $j\ge 0$.

For $e\in\R^d$, $|e|=1$,
$\sigma>0$,
we consider the sequences of functions $u^{e,\sigma}_k:=t_\sigma(e\cdot (u_k-v_k))$. We let $\omega_k^j:=\bigcup_{Q\in\Good^{\delta_j}(u_k)}\omega_k^Q$,
then $|\omega_k^j|\le c \delta_j\H^{d-1}(J_{u_k}^1)$. Thus:
\begin{equation}\label{eq:cauchy}
  \begin{aligned}
    \int_\Omega |u^{e,\sigma}_k-t_\sigma(e\cdot w_k^j)|dx
    \le\ &\sigma \left|B_j\cup \omega_k^j\right|
    +\int_{\Omega\setminus (B_j\cup \omega_k^j)}|e\cdot(u_k-w_k^j)|dx\\
    \le\ &\sigma \eta_j + \sum_{Q\in\Good^{\delta_j}(u_k)} \int_{Q\setminus \omega_k^Q} |u_k-a_k^Q|dx\\
    \le\ & \sigma \eta_j + C\delta_j
  \end{aligned}
\end{equation}
where we have let $\eta_j=|B_j|+ c\delta_j\sup_k\H^{d-1}(J_{u_k}^1)$,
and $C=c\sup_k \widehat{\mu}_{u_k}(\Omega\setminus J_{u_k}^1)<+\infty$,
and used~\eqref{eq:rigid}.
Using that $w_k^j-w_l^j\to
0$ as $k,l\to\infty$ in $L^1(\Omega)$, we find that:
\[
  \limsup_{k,l\to\infty} \int_\Omega |u^{e,\sigma}_k-u^{e,\sigma}_l|dx\le
  2(\sigma\eta_j + C\delta_j).
\]
Sending $j\to +\infty$ we find that $\limsup_{k,l\to\infty} \int_\Omega |u^{e,\sigma}_k-u^{e,\sigma}_l|dx=0$,
that is, $(u^{e,\sigma}_k)_k$ are Cauchy sequences and converge to some limit $u^{e,\sigma}$
in $L^1(\Omega)$.

We show now that the limit $u^{e,\sigma}$ is $t_\sigma(e\cdot u)$ for some well-defined
measurable function $u$.
Let us consider a subsequence such that $u^{e,1}_k\to u^{e,1}$ a.e.
In that case:
\begin{itemize}
\item  either $|u^{e,1}(x)|=1$, which happens if and only if
  $\lim_{k\to\infty} |e\cdot (u_k(x)-v_k(x)) |= +\infty$,
  and in particular for any $\sigma>0$,  $|u^{e,\sigma}(x)|=\sigma$;
\item or, by continuity, $e\cdot (u_k(x)-v_k(x))
  \to \tanh^{-1} (u^{e,1}(x))$, and we also have that $u^{e,\sigma}(x)= t_\sigma(\tanh^{-1} (u^{e,1}(x)))$ for any $\sigma>0$.
\end{itemize}
Let $A= \{x:|u^{e,1}(x)|=1\}=\{x:|u^{e,\sigma}(x)|=\sigma\}$: then, for $j\ge 0$,
\[
  \int_{A\setminus B_j} |u^{e,\sigma}_k-t_\sigma(e\cdot w_k^j)|dx \ge
  \int_{A\setminus B_j} |u^{e,\sigma}_k|dx -\int_{A\setminus B_j} |w_k^j|dx
  \stackrel{k\to\infty}{\longrightarrow}
  \sigma|A\setminus B_j| - \int_{A\setminus B_j} |w^j|dx.
\]
On the other hand thanks to~\eqref{eq:cauchy}:
\[
  \int_{A\setminus B_j} |u^{e,\sigma}_k-t_\sigma(e\cdot w_k^j)|dx \le\sigma \eta_j+C\delta_j,
\]
hence:
\[
  \sigma|A\setminus B_j| \le  \int_{A\setminus B_j} |w^j|dx + \sigma \eta_j+C\delta_j.
\]
Dividing by $\sigma$ and letting $\sigma\to\infty$, we deduce:
\[
  |A\setminus B_j|\le \eta_j,
\]
so that $|A|=0$. It follows that $\tanh^{-1} u^{e,1}$ is finite a.e.

To sum up, we have shown that for any $e$ with $|e|=1$, there is a
measurable function $u^e$ such that $t_\sigma(e\cdot (u_k-v_k))\to t_\sigma(u^e)$
in $L^1(\Omega)$ for any $\sigma>0$. It is then obvious to check that $u^e=e\cdot u$
for some measurable vector-valued function $u$, and, up to a subsequence, to
deduce that $u_k-v_k\to u$ a.e.~in $\Omega$.

\subsection{Lower semicontinuity}\label{sec:lsc}
We argue similarly to what done in \cite[Step~2 in Section~3]{ChaCri20AnnSNS} to prove the $GSBD^p$ analogue of \eqref{eq:sciSaltoInf}, that is \cite[equation~(1.5d)]{ChaCri20AnnSNS}.

Let us fix $\sigma>1$, $\xi \in \Sn$ in a set of full $\hn$-measure of $\Sn$ for which \eqref{2502200925} holds (\textit{cf.}~Lemma~\ref{le:2502200924}), and define
\begin{equation}\label{0101182132}
\I(u_k):=|\mathrm{D}(\widehat{u}_k)\xy|\big(\Omega\xy \sm J^\sigma_{(\widehat{u}_k)\xy}\big)\,.
\end{equation}
Since 
\[
\I(u_k)\leq (\widehat{\mu}_{u_k})\xy(\Omega\xy)+(\sigma-1) \mathcal{H}^0\big(\Omega\xy \cap (J^1_{(\widehat{u}_k)\xy}\sm J^\sigma_{(\widehat{u}_k)\xy})\big)\leq \sigma (\widehat{\mu}_{u_k})\xy(\Omega\xy),
\]
it holds that
\begin{equation}\label{0908221020}
\int_{\Pi^\xi} \I(u_k) \dh(y) \leq \sigma \widehat{\mu}^\xi_{u_k}(\Omega) \leq \sigma \sup_{k\in \N} \widehat{\mu}_{u_k}(\Omega).
\end{equation}
Following exactly \cite[Step~2 in Section~3]{ChaCri20AnnSNS} for $\sigma>1$ fixed (using \eqref{0908221020} in place of \cite[estimate (3.11)]{ChaCri20AnnSNS}), we get that all the \cite[(3.12)-(3.18)]{ChaCri20AnnSNS} hold for $J_v$ replaced by $J^\sigma_v$ and $\widehat{I}\xy$ replaced by $\I$. In particular, for $\hn$-a.e.\ $\xi \in \Sn$ and $\hn$-a.e.\ $y\in \Pi_\xi$, along a suitable subsequence $(\cdot)_j$ depending on $\sigma$, $\xi$, $\varepsilon \in (0,\sigma^{-2})$ fixed as in \cite[(3.13)]{ChaCri20AnnSNS}, and $y$, it holds that 
\begin{equation}\label{1203220933}
(\widehat{u}_j-\widehat{a}_j)\xy \to \widehat{u}\xy \quad\mathcal{L}^1\text{-a.e.\ in }\Omega\xy,
\end{equation}
\begin{equation}\label{2403201949}
 |( \widehat{a}^{i_1}_j-\widehat{a}^{i_2}_j)\xy(t)|=|( \widehat{a}^{i_1}_j-\widehat{a}^{i_2}_j)\xy(0)| \to +\infty \quad\text{for }t\in \Omega\xy \text{ and }i_1 \neq i_2,
\end{equation}
\begin{equation}\label{0101182346}
\lim_{j\to \infty} \Big( \mathcal{H}^0\big(J^\sigma_{(\widehat{u}_j)\xy}\big) + \varepsilon \I(u_j)\Big)=\liminf_{m\to \infty} \Big( \mathcal{H}^0\big(J^\sigma_{(\widehat{u}_m)\xy}\big) + \varepsilon \I(u_m) \Big) =M(y) \in \R
\end{equation}
for $(\cdot)_m$ a subsequence of $(\cdot)_k$ independent of $y$ such that 
\begin{equation}\label{0908221224}
H^\sigma_{\varepsilon,\xi}:=\lim_{m\to\infty} \int_{\Pi_\xi} \Big( \mathcal{H}^0\big(J^\sigma_{(\widehat{u}_m)\xy}\big) + \varepsilon \I(u_m) \Big) \dh(y)\in \R
\end{equation}
and
\begin{equation}\label{0908221224'}
\limsup_{\varepsilon\to 0}\int_{\Sn} H^\sigma_{\varepsilon,\xi}\dh(\xi)\leq \liminf_{k\to \infty} \hn(J^\sigma_{u_k}).
\end{equation}
Denoting 
\begin{equation*}
\partial \mathcal{P}\xy:= \bigcup_{n\in \N} \partial \left((P_n)\xy\right)  \cap \Omega\xy  \subset \Omega\xy,
\end{equation*}
by \eqref{0101182346} we may assume, up to a further subsequence, that $\mathcal{H}^0\big(J^\sigma_{(\widehat{u}_j)\xy}\big)=N_y\in \N$ for every $j$ and so that there are $\widehat{N}_y\leq N_y$ cluster points in the limit, denoted by 
\[t_1,\dots ,t_{\widehat{N}_y}.\]
Therefore we have that $K \cap J^\sigma_{(\widehat{u}_j)\xy}=\emptyset$ for any $K$ compact subset of $(t_l, t_{l+1})$, so $|\mathrm{D}(\widehat{u}_k)\xy|\big(K \sm J^\sigma_{(\widehat{u}_j)\xy}\big)=|\mathrm{D}(\widehat{u}_j)\xy|(K)$; with \eqref{0101182346} and the Fundamental Theorem of Calculus, this gives that,
for $\mathcal{L}^1$-almost any choice of $ \ol t \in (t_l, t_{l+1})$, 
\begin{equation}\label{2502201338}
 t  \mapsto (\widehat{u}_j)\xy( t ) - (\widehat{u}_j)\xy( \ol t ) \text{ are equibounded  w.r.t.\  $j$ in } BV_\mathrm{loc}(t_l, t_{l+1}),
\end{equation} 
so the bound above is also in $L^\infty_{\mathrm{loc}}(t_l, t_{l+1})$. 
At this stage we prove, as in \cite[(3.20)]{ChaCri20AnnSNS}, that
 \begin{equation}\label{2502201307}
 \partial \mathcal{P}\xy \subset \{ t_1,\dots ,t_{\widehat{N}_y}\}:
\end{equation}  
in fact, assuming by contradiction that there exists $l\in \{1,\dots, M_y\}$ and $i_1\neq i_2$ such that $\partial (P_{i_1})\xy \cap (t_l, t_{l+1}), \partial (P_{i_2})\xy \cap (t_l, t_{l+1})\neq \emptyset$, by \eqref{1203220933} there are two corresponding sequences of infinitesimal rigid motions $(a_j^{i_1})_j$, $(a_j^{i_2})_j$ for which
\begin{equation}\label{2502201322}
\begin{split}
&(\widehat{u}_j - \widehat{a}_j^{i_1})\xy \to \widehat{u}\xy \quad\mathcal{L}^1\text{-a.e.\ in } (P_{i_1})\xy \cap (t_l, t_{l+1}), \\
& (\widehat{u}_j - \widehat{a}_j^{i_2})\xy\to \widehat{u}\xy\quad\mathcal{L}^1\text{-a.e.\ in } (P_{i_2})\xy \cap (t_l, t_{l+1}),
\end{split}
\end{equation}
with $\mathcal{L}^1\big(  (P_{i_1})\xy  \cap (t_l, t_{l+1})  \big),\, \mathcal{L}^1\big(  (P_{i_2})\xy  \cap (t_l, t_{l+1})  \big)>0$; this gives (with \eqref{2502201338} and since $\widehat{a}_j^i$ are infinitesimal rigid motions and $\widehat{u}\xy\colon \Omega\xy \to \R$) that 
$( \widehat{a}^{i_1}_j-\widehat{a}^{i_2}_j)\xy $ is constant in $\Omega\xy$ and uniformly bounded   w.r.t.\  $j$, in contradiction with \eqref{2403201949}. Therefore, \eqref{2502201307} is proven. 

Integrating \eqref{2502201307} over $y\in\Pi^\xi$ and using Fatou's lemma with \eqref{0101182346}, \eqref{0908221224} we deduce that
\begin{equation}\label{0201181712}
\begin{split}
\int \limits_{\Pi^\xi} \ho(\partial \mathcal{P}\xy) \,\mathrm{d}\hn(y) \leq \lim_{m\to\infty} \int\limits_{\Pi^\xi}\Big( \ho\big(J^\sigma_{(\widehat{u}_m)\xy}\big) + \varepsilon \, \I(u_m) \Big)\,\mathrm{d}\hn(y) 
\end{split}
\end{equation}
for every $\sigma>1$ and $\hn$-a.e.\ $\xi \in \Sn$ (in view of the choice of the subsequences, see \cite{ChaCri20AnnSNS}). This implies that $\mathcal{P}$ is a Caccioppoli partition.

Moreover, integrating over $\xi \in \Sn$ we get that
\begin{equation}\label{0201181712'}
\begin{split}
\hn(\partial^*\mathcal{P}\cap \Omega) &\leq  C\varepsilon \sigma \sup_{k\in \N} \widehat{\mu}_{u_k}(\Omega) + \int_{\Sn} H^\sigma_{\varepsilon,\xi}\dh(\xi) \\& \leq C\sqrt{\varepsilon} \sup_{k\in \N} \widehat{\mu}_{u_k}(\Omega) + \int_{\Sn} H^\sigma_{\varepsilon,\xi}\dh(\xi)
\end{split}
\end{equation}
for a universal constant $C>0$ and every $\sigma>1$, $\varepsilon \in (0, \sigma^{-2})$.
Letting $\varepsilon\to 0$, in view of \eqref{0908221224'} and the arbitrariness of $\sigma>1$ we conclude \eqref{eq:sciSaltoInf}.

Let us now confirm that $u \in GBD(\Omega)$.
For any $\xi \in \Sn$ and $\hn$-a.e.\ $y\in \Pi^\xi$, setting $\tilde{u}_k:=u_k-a_k$ for any $k\in \N$, it holds that
\begin{equation*}
(\widehat{\mu}_{\tilde{u}_k})\xy(B)\leq(\widehat{\mu}_{u_k})\xy(B) + \ho(\partial^* \mathcal{P}\xy \cap B) \quad \text{for every $B\subset \Omega\xy$ Borel},
\end{equation*}
since $a_k$ is a piecewise rigid motion constant on every $P_j$, $j\in \N$, where $\mathcal{P}=(P_j)_j$. Integrating over $\Pi_\xi$ and recalling \eqref{1707221155} we have that
\begin{equation*}
\widehat{\mu}_{\tilde{u}_k}(B)\leq \widehat{\mu}_{u_k}(B) + \hn(\partial^* \mathcal{P} \cap B) \quad \text{for every $B\subset \Omega$ Borel}.
\end{equation*}
Since $\tilde{u}_k$  pointwise converges $\Ln$-a.e.\ to $u$, there is an increasing function $\psi_0\colon \R^+\to \R^+$ with $\lim_{s\to +\infty}\psi_0(s)=+\infty$ such that $\|\psi_0(\tilde{u}_k)\|_{L^1(\Omega)}$ is uniformly bounded w.r.t.\ $k\in \N$ (see e.g.\ \cite[Lemma~2.1]{Fri19}).
Then we may apply \cite[Corollary~11.2]{DM13} to deduce that $u \in GBD(\Omega)$.

This concludes the proof of Theorem~\ref{thm:main}.

\begin{remark}\label{rem:GSBDp}
Let us consider a sequence $(u_k)_k$ such that
\begin{equation}\label{eq:bdGSBDp}
\int_\Omega |e(u_k)|^p \dx + \hn(J_{u_k}) \leq M, \quad p>1
\end{equation}
for $M>0$ independent of $k\in \N$. Applying Theorem~\ref{thm:main} we obtain the compactness part of \cite[Theorem~1.1]{ChaCri20AnnSNS} (that is \cite[(1.5b)]{ChaCri20AnnSNS} for $u \in GBD(\Omega)$) without using the Korn-type inequality in \cite{CagChaSca20}. Combining this with the last part of the proof (\cite[Steps~2-3 in proof of Theorem~1.1]{ChaCri20AnnSNS}) we obtain \cite[Theorem~1.1]{ChaCri20AnnSNS}, and
in particular~\eqref{eq:sciSaltoInfp}.
Nevertheless, the result in \cite{CagChaSca20} is crucial for \cite[Theorem~1.2]{ChaCri20AnnSNS}.
\end{remark}



\begin{thebibliography}{10}

\bibitem{AlmiTasso}
{\sc S.~Almi and E.~Tasso}, {\em A new proof of compactness in {$G(S)BD$}},
  Advances in Calculus of Variations,  (2022).

\bibitem{AFP}
{\sc L.~Ambrosio, N.~Fusco, and D.~Pallara}, {\em Functions of bounded
  variation and free discontinuity problems}, Oxford Mathematical Monographs,
  The Clarendon Press, Oxford University Press, New York, 2000.

\bibitem{BelCosDM98}
{\sc G.~Bellettini, A.~Coscia, and G.~Dal~Maso}, {\em Compactness and lower
  semicontinuity properties in {${\rm SBD}(\Omega)$}}, Math. Z., 228 (1998),
  pp.~337--351.

\bibitem{CagChaSca20}
{\sc F.~Cagnetti, A.~Chambolle, and L.~Scardia}, {\em Korn and
  {P}oincar\'{e}-{K}orn inequalities for functions with a small jump set},
  Math. Ann., 383 (2022), pp.~1179--1216.

\bibitem{CCF16}
{\sc A.~Chambolle, S.~Conti, and G.~Francfort}, {\em Korn-{P}oincar\'{e}
  inequalities for functions with a small jump set}, Indiana Univ. Math. J., 65
  (2016), pp.~1373--1399.

\bibitem{CCI19}
{\sc A.~Chambolle, S.~Conti, and F.~Iurlano}, {\em Approximation of functions
  with small jump sets and existence of strong minimizers of {G}riffith's
  energy}, J. Math. Pures Appl. (9), 128 (2019), pp.~119--139.

\bibitem{CC-CalcVar}
{\sc A.~Chambolle and V.~Crismale}, {\em Existence of strong solutions to the
  {D}irichlet problem for the {G}riffith energy}, Calc. Var. Partial
  Differential Equations, 58 (2019), pp.~Paper No. 136, 27.

\bibitem{CC-JEMS}
\leavevmode\vrule height 2pt depth -1.6pt width 23pt, {\em Compactness and
  lower semicontinuity in {$GSBD$}}, J. Eur. Math. Soc. (JEMS), 23 (2021),
  pp.~701--719.

\bibitem{ChaCri20AnnSNS}
\leavevmode\vrule height 2pt depth -1.6pt width 23pt, {\em Equilibrium
  configurations for nonhomogeneous linearly elastic materials with surface
  discontinuities}, Ann.~SNS Pisa Cl.~Sci.,  (2022).
\newblock (preprint hal-02667936, to appear).

\bibitem{CFI19}
{\sc S.~Conti, M.~Focardi, and F.~Iurlano}, {\em Existence of strong minimizers
  for the {G}riffith static fracture model in dimension two}, Ann. Inst. H.
  Poincar\'{e} C Anal. Non Lin\'{e}aire, 36 (2019), pp.~455--474.

\bibitem{DM13}
{\sc G.~Dal~Maso}, {\em Generalised functions of bounded deformation}, J. Eur.
  Math. Soc. (JEMS), 15 (2013), pp.~1943--1997.

\bibitem{DGA}
{\sc E.~De~Giorgi and L.~Ambrosio}, {\em New functionals in the calculus of
  variations}, Atti Accad. Naz. Lincei Rend. Cl. Sci. Fis. Mat. Nat. (8), 82
  (1988), pp.~199--210 (1989).

\bibitem{DelNin}
{\sc G.~Del~Nin}, {\em Rectifiability of the jump set of locally integrable
  functions}, Ann. Sc. Norm. Super. Pisa Cl. Sci. (5), 22 (2021),
  pp.~1233--1240.

\bibitem{FM}
{\sc G.~A. Francfort and J.-J. Marigo}, {\em Revisiting brittle fracture as an
  energy minimization problem}, J. Mech. Phys. Solids, 46 (1998),
  pp.~1319--1342.

\bibitem{FriPWKorn}
{\sc M.~Friedrich}, {\em A piecewise {K}orn inequality in {$SBD$} and
  applications to embedding and density results}, SIAM J. Math. Anal., 50
  (2018), pp.~3842--3918.

\bibitem{Fri19}
\leavevmode\vrule height 2pt depth -1.6pt width 23pt, {\em A compactness result
  in {$GSBV^p$} and applications to {$\Gamma$}-convergence for free
  discontinuity problems}, Calc. Var. Partial Differential Equations, 58
  (2019), pp.~Paper No. 86, 31.

\bibitem{KohnTemam}
{\sc R.~Kohn and R.~Temam}, {\em Dual spaces of stresses and strains, with
  applications to {H}encky plasticity}, Appl. Math. Optim., 10 (1983),
  pp.~1--35.

\bibitem{Nitsche81}
{\sc J.~A. Nitsche}, {\em On {K}orn's second inequality}, RAIRO Anal. Num\'er.,
  15 (1981), pp.~237--248.

\bibitem{Suquet1981}
{\sc P.-M. Suquet}, {\em Sur les \'equations de la plasticit\'e: existence et
  r\'egularit\'e des solutions}, J. M\'ecanique, 20 (1981), pp.~3--39.

\bibitem{Tem}
{\sc R.~Temam}, {\em Mathematical problems in plasticity}, Gauthier-Villars,
  Paris, 1985. Translation of Problèmes mathématiques en plasticité.
  Gauthier-Villars, Paris, 1983.

\bibitem{TemStr}
{\sc R.~Temam and G.~Strang}, {\em {Duality and relaxation in the variational
  problem of plasticity}}, J. Mécanique, 19 (1980), pp.~493--527.

\end{thebibliography}

\end{document}